\documentclass[1p]{elsarticle}
\usepackage[utf8]{inputenc}
\usepackage{amsmath,amssymb,amscd,amsthm}
\usepackage[colorlinks]{hyperref}
\hypersetup{
	linkcolor=black,
	citecolor=black,
	pdftitle={},
	pdfauthor={Franziska K\"uhn, Ren\'e Schilling},
	pdfkeywords={},
}

\newtheorem{theorem}{Theorem}[section]
\newtheorem{lemma}[theorem]{Lemma}
\newtheorem{corollary}[theorem]{Corollary}
\newtheorem{proposition}[theorem]{Proposition}

\newdefinition{remark}[theorem]{Remark}
\newdefinition{example}[theorem]{Example}
\newdefinition{definition}[theorem]{Definition}

\newcommand{\re}{\operatorname{\mathrm{Re}}}

\newcommand{\Bcal}{\mathcal{B}}
\newcommand{\Fcal}{\mathcal{F}}

\newcommand{\rk}{{\mathbb{R}^k}}
\newcommand{\real}{\mathbb{R}}
\newcommand{\nat}{\mathbb{N}}

\newcommand{\Ee}{\mathbb{E}}
\newcommand{\Pp}{\mathbb{P}}

\newcommand{\I}{\mathbf{1}}

\newcommand{\cm}{\mathcal{CM}}

\newcommand\fa{\quad\text{for all \ }}

\begin{document}
	
\title{A probabilistic proof of Schoenberg's theorem}
	
\author[FK]{Franziska K\"{u}hn\corref{cor1}}
\cortext[cor1]{Corresponding author}
\ead{franziska.kuhn@math.univ-toulouse.fr}
\address[FK]{Institut de Math\'ematiques de Toulouse, Universit\'e Paul Sabatier III Toulouse, 118 Route de Narbonne, 31062 Toulouse, France}
	
\author[RS]{Ren\'e L.\ Schilling}
\ead{rene.schilling@tu-dresden.de}
\address[RS]{TU Dresden, Fakult\"{a}t Mathematik, Institut f\"{u}r Mathematische Stochastik, 01062 Dresden, Germany}

		
\begin{abstract}
    Assume that $g(|\xi|^2)$, $\xi\in\rk$, is for every dimension $k\in\nat$ the characteristic function of an infinitely divisible random variable $X^k$. By a classical result of Schoenberg $f:=-\log g$ is a Bernstein function. We give a simple probabilistic proof of this result starting from the observation that $X^k = X_1^k$ can be embedded into a L\'evy process $(X_t^k)_{t\geq 0}$ and that Schoenberg's theorem says that $(X_t^k)_{t\geq 0}$ is subordinate to a Brownian motion. A key ingredient of our proof are concrete formulae which connect the transition densities, resp., L\'evy measures of subordinated Brownian motions across different dimensions. As a by-product of our proof we obtain a gradient estimate for the transition semigroup of a subordinated Brownian motion.
\end{abstract}
	
\maketitle
		
\section{Introduction}
I.J.\ Schoenberg proved in 1938 \cite{schoen} the following 
\begin{quote}
    \textbf{Theorem A:} {\itshape If $\rk \ni (x_1,\dots,x_k)\mapsto g(x_1^2+\dots+x_k^2)$ is positive definite \textup{(}in the sense of Bochner\textup{)} for any dimension $k\in\nat$, then $g(r)$, $r>0$, is a completely monotone function.}
\end{quote} 
Recall that $g: (0,\infty) \to (0,\infty)$ is \emph{completely monotone} (notation: $g \in \cm$), if
\begin{equation}\label{lp-eq5}
    g\in C^\infty
    \quad\text{and}\quad
    (-1)^n g^{(n)}(r) \geq 0 \fa r>0, \; n \in \nat_0.
\end{equation}
Schoenberg used Theorem A to determine all positive definite functions in a Hilbert space; this was part of his programme to characterize all metrics $\rho$ in $\rk$ such that $(\rk,\rho)$ can be isometrically embedded into a Hilbert space, cf.\ \cite{schoen2}. A necessary and sufficient condition turns out to be that $\rk \ni x\mapsto e^{-t \rho^2(x,0)}$ is positive definite; in other words: all such metrics are of the form $\rho(x,y) = \sqrt{\psi(x-y)}$ where $\psi$ is the (non-negative!) characteristic exponent of a symmetric L\'evy process. This allows us to re-cast Schoenberg's theorem in the form of 
\begin{quote}
    \textbf{Theorem B:} {\itshape If $\rk \ni (x_1,\dots,x_k)\mapsto \exp\left[-tf(x_1^2+\dots+x_k^2)\right]$ is positive definite \textup{(}in the sense of Bochner\textup{)} for all $t>0$ and any dimension $k\in\nat$, then $g(r) = e^{-tf(r)}$, $r>0$, is a completely monotone function.}
\end{quote}
In probabilistic terms, this means that $g(x_1^2+\dots+x_k^2)$ is an infinitely divisible characteristic function and $f: (0,\infty)\to (0,\infty)$ is a \emph{Bernstein function}, i.e.
\begin{equation}\label{lp-eq4}
    f\in C^\infty,\quad f \geq 0
    \quad\text{and}\quad
    (-1)^{n-1} f^{(n)}(r) \geq 0 \fa r>0, \; n\in\nat.
\end{equation}

Both theorems have attracted a lot of attention and there are several proofs highlighting various (hidden) aspects of Schoenberg's result. Let us briefly describe some of the developments. Modern (streamlined) versions of the classical proof of Theorem A can be found in Donoghue \cite[p.~205]{dono} and Steerneman \& van Perlo-ten Kleij \cite{steern}, where the presentation of the convergence argument as $k\to\infty$ is simplified; following Bochner \cite[p.~99]{bochner}, Theorem B is e.g.\ proved in \cite[Theorem 13.14]{bernstein}.

Using Bochner's characterization of positive definite functions and the solution of Hausdorff's moment problem, Ressel \cite{ressel76} proves Theorem A in a completely different way. This approach is generalized to semigroups in Berg \emph{et al.} \cite[Chapter~5]{berg}. Independent of Ressel, Kahane \cite{kahane} uses essentially the same argument to prove both Theorem A and B. Combining Bochner's theorem with the characterization of completely monotone functions by iterated differences\footnote{In the end, this characterization relies on a deep application of the Krein--Milman theorem, cf.\ \cite[Theorem~4.8]{bernstein}.}, Wendland \cite[Theorem~7.13]{wendland} gives a short proof of Theorem A which is inspired by earlier work by Kuelbs \cite{kuelbs} and Wells \& Williams~\cite[Chapter~II]{wells-williams}. Let us point out that the essential step in these proofs, \cite[p.~94, last 4 lines]{wendland} (also \cite[Lemma~2.1]{kuelbs}, \cite[Theorem~7.2]{wells-williams}) is already present in Harzallah's proof that Bernstein functions operate on negative definite functions \cite[Lemma~6]{harzallah67}, see also Jacob \cite[Lemma~3.9.22]{jacob} who works out this detail.

The methods to prove Theorems A and B are closely related to the so-called \emph{Schoenberg's problem} in the geometry of Banach spaces: 
\begin{quote}
    {\itshape Determine all values $\alpha\geq 0$ such that $\exp(-\|x\|_{\ell^p}^\alpha)$ is positive definite on $\real^k$ with $k\geq 2$ and $p\geq 1$.} 
\end{quote}
For $1\leq p\leq 2$ this is discussed by Bretagnolle \emph{et al.}\ \cite{bret66} (who establish the connection with the embeddability of normed linear spaces into $L^p$); Zastavnyi \cite{zast} has the definitive solution.

Our approach to prove Theorem B uses elements of the Fourier approach from the original proof of Schoenberg's theorem, but the rather awkward limiting argument, sending the dimension $k\to\infty$, is now replaced by a ``dimension walk'' argument which was pioneered by Matheron who calls it the \emph{mont\'ee et descente en clavier isotrope} \cite[pp.~31--37]{matheron65}, see also the unpublished manuscript \cite{matheron72}.

\section{Preliminaries}
A function $u: \rk \to \real$ is called \emph{rotationally invariant} if $u(x)$ depends only on $|x|$, i.e.\ if $u(x) = U(|x|)$ for some function $U: [0,\infty) \to \real$. In abuse of notation we write $u(r) = U(r)$ for $r \geq 0$. For an integrable function $u: \rk \to \real$ we denote by
\begin{align} \label{ft-eq1}
\begin{aligned}
	\Fcal_k u(\xi) &:= \frac{1}{(2\pi)^k} \int_{\rk} e^{-ix \cdot \xi} u(x) \, dx, & \xi \in \rk, \\
	\Fcal_k^{-1} u(\xi) &:= \int_{\rk} e^{ix \cdot \xi} u(x) \, dx, & \xi \in \rk,
\end{aligned}
\end{align}
the \emph{Fourier transform} and \emph{inverse Fourier transform} of $u$, respectively. If $u$ is rotationally invariant, then both $\Fcal_k u$ and $\Fcal_k^{-1} u$ are rotationally invariant and \begin{align}
	\Fcal_k u(r)
	= \frac{1}{(2\pi)^k} \Fcal_k^{-1} u(r)
	= \frac{1}{(2\pi)^{k/2} r^{k/2-1}} \int_{(0,\infty)} u(s) s^{k/2} J_{k/2-1}(sr) \, ds \label{ft-eq2}
\end{align}
where $J_{\nu}$ denotes the Bessel function of the first kind, see e.g.\ \cite[Example 19.4]{mims} or \cite[Theorem 5.26]{wendland} for a proof. Using \eqref{ft-eq2} and some identities for Bessel functions \cite[(10.6.2)]{nist} (see also the proof of Theorem~\ref{intro-9}), it is not hard to see that
\begin{equation}\label{ft-eq3}
		\Fcal_{k+2} u(r) = - \frac{1}{2\pi} \frac{1}{r} \frac{d}{dr} \Fcal_k u(r)
\end{equation}
for any rotationally invariant function such that $u(|\cdot|) \in L^1(\rk,dx) \cap L^1(\real^{k+2},dx)$; this observation is due to Matheron \cite[(1.4.9)]{matheron65}.

Let $(\Omega,\mathcal{A},\Pp)$ be a probability space. A random variable $X: \Omega \to \rk$ is called \emph{unimodal isotropic} if $\Pp(X \in dx) = c \delta_0(dx) + p(|x|) \, dx$ for some non-increasing $p:(0,\infty) \to [0,\infty)$ and $c \in [0,1]$.  A family of random variables $X_t: \Omega \to \rk$ is called (\emph{$k$-dimensional}) \emph{L\'evy process} if $X_0 = 0$ a.s., $(X_t)_{t \geq 0}$ has independent and stationary increments and $t \mapsto X_t(\omega)$ is for almost all $\omega \in \Omega$ right-continuous with finite left limits. Our standard reference for L\'evy processes is the monograph by Sato \cite{sato}. For an introduction to L\'evy processes we also recommend \cite{barca}. We will often use the superscript to indicate the dimension, i.e.\ we write $(X_t^k)_{t \geq 0}$ for a $k$-dimensional L\'evy process. It is well known, cf.\ \cite{sato}, that $(X_t)_{t \geq 0}$ can be uniquely characterized via its \emph{characteristic exponent},
\begin{equation*}
    \psi(\xi)
    = - i \, b \cdot \xi + \frac{1}{2} \xi \cdot Q \xi + \int_{\rk \setminus \{0\}} \left(1-e^{i y \cdot \xi}+ i \, y \cdot \xi \I_{(0,1)}(|y|)\right) \, \nu(dy),
    \quad \xi \in \rk;
\end{equation*}
the \emph{L\'evy triplet} $(b,Q,\nu)$ consists of the \emph{drift} $b \in \rk$, a positive semi-definite symmetric matrix $Q \in \real^{k \times k}$ and the \emph{L\'evy measure} $\nu$ on $(\rk \setminus \{0\},\Bcal(\rk \setminus \{0\}))$ satisfying $\int_{\rk \setminus \{0\}} \min\{1,|y|^2\} \, \nu(dy)<\infty$. We say that $\psi$ satisfies the \emph{Hartman--Wintner condition} if
\begin{equation}\tag{HW}\label{lp-eq2}
	\lim_{|\xi| \to \infty} \frac{\re \psi(|\xi|)}{\log |\xi|} = \infty.
\end{equation}
It is shown in \cite{knop} that the Hartman--Wintner condition is equivalent to the existence of a smooth transition density $p_t$ for all $t>0$. A function $\psi$ is \emph{continuous negative definite} (in the sense of Schoenberg) if, and only if, it is the characteristic exponent of a L\'evy process. The domain of the \emph{generator} $A=A_k$ of a $k$-dimensional L\'evy process contains the compactly supported smooth functions $C_c^{\infty}(\rk)$ and
\begin{equation}\label{lp-eq1}
	A_k u(x)
    = - \int_{\rk} e^{ix \cdot \xi} \psi(\xi) \Fcal_k u(\xi) \, d\xi
    = -\Fcal_k^{-1}(\psi \cdot \Fcal_k u)(x), \quad u \in C_c^{\infty}(\rk),\; x \in \rk,
\end{equation}
cf.\ \cite[Theorem 6.8]{barca}. If $\psi$ is a rotationally invariant characteristic exponent of a $k$-dimensional L\'evy process  and $u(x) = u(|x|)$ a rotationally invariant function with compact support, then we write in accordance with \eqref{ft-eq2} \begin{equation}
	A_k u(r) := - \Fcal_k^{-1}(\psi \cdot \Fcal_k u)(r), \quad r \geq 0. \label{ft-eq4}
\end{equation}
The \emph{jump measure} $N$ of $(X_t)_{t \geq 0}$ is given by
\begin{equation}\label{lp-eq3}
	N_t(B) := \# \{s \in [0,t]; \Delta X_s := X_s-X_{s-} \in B\}, \quad B \in \mathcal{B}(\rk \setminus \{0\}),\; t \geq 0. \end{equation}
For any fixed Borel set $B \in \mathcal{B}(\rk \setminus \{0\})$ the process $(N_t(B))_{t \geq 0}$ is a Poisson process with intensity $\nu(B)$, cf.\ \cite[Lemma 9.4]{barca}.

A one-dimensional L\'evy process $(S_t)_{t \geq 0}$ is called a \emph{subordinator} if $(S_t)_{t \geq 0}$ has non-decreasing sample paths. A subordinator is uniquely characterized by its Laplace transform $\Ee e^{-u S_t} = e^{-t f(u)}$, $u \geq 0$; the characteristic (Laplace) exponent $f$ is a \emph{Bernstein function}, i.e.
\begin{equation*}
	f(u) = \alpha u + \int_0^{\infty} (1-e^{-uy}) \, \mu(dy), \quad u \geq 0,
\end{equation*}
for $\alpha \geq 0$ and a measure $\mu$ on $(0,\infty)$ such that $\int_{0}^{\infty} \min\{1,y\} \, \mu(dy) <\infty$. By Bernstein's theorem, cf.\ \cite[Theorem 3.2]{bernstein}, this is equivalent to \eqref{lp-eq4}.

If $(S_t)_{t \geq 0}$ is a subordinator with Laplace exponent $f$ and $(B_t)_{t \geq 0}$ an independent Brownian motion, then the \emph{subordinated Brownian motion} $(B_{S_t})_{t \geq 0}$ is again a  L\'evy process, and its characteristic exponent is given by $\psi(\xi) = f(|\xi|^2)$. A comprehensive treatment of completely monotone functions, Bernstein functions and subordination is given in \cite{bernstein}.

\section{Results}

We will prove the following extended version of Schoenberg's theorem.
\begin{theorem}\label{intro-5}
    Let $f:[0,\infty) \to [0,\infty)$. The following statements are equivalent.
    \begin{enumerate}[(i)]
	\item\label{intro-5-i}
        $\rk \ni \xi \mapsto f(|\xi|^2)$ is a continuous negative definite function for all $k \geq 1$.
	\item\label{intro-5-ii}
        $\rk \ni \xi \mapsto f(|\xi|^2)$ is a continuous negative definite function for any $k=1+2n$, $n \in \mathbb{N}_0$.
	\item\label{intro-5-iii}
        $f$ is a Bernstein function.
	\item\label{intro-5-iv}
        For some \textup{(}all\textup{)} $k \geq 1$ there exists a $k$-dimensional L\'evy process $(X_t)_{t \geq 0}$ with characteristic exponent $\psi(\xi) := f(|\xi|^2)$ and
        \begin{equation}\label{intro-eq11}
		  \Pp(X_t \in B) = e^{-ct} \delta_0(B) + \int_B p_t(x) \, dx, \quad B \in \Bcal(\rk),\; t>0,
		\end{equation}
        for some constant $c \in [0,\infty]$ and a rotationally  invariant function $p_t: \rk \to [0,\infty)$ such that $p_t(\sqrt{\cdot})$ is completely monotone.
	\end{enumerate}
\end{theorem}
The proof of Theorem~\ref{intro-5} actually shows that $f$ is a Bernstein function if, and only if, there exist infinitely many $n \in \nat$ such that $\rk \ni \xi \mapsto f(|\xi|^2)$ is a continuous negative definite function for $k=n$ and $k=n+2$. Moreover, if $(S_t)_{t \geq 0}$ is a subordinator with Laplace exponent $f$, then the L\'evy process $(X_t)_{t \geq 0}$ is subordinate to a Brownian motion, i.e.\ it is, in distribution, a time-changed Brownian motion $(B_{S_t})_{t \geq 0}$. The next corollary reveals how the density function $p_t^k$ and the L\'evy measure $\nu_k$ in different dimensions are related. As before, we use sub- and superscripts to indicate the dimension. We continue using the notation of Theorem~\ref{intro-5}.

\begin{corollary} \label{intro-6}
    Let $f:[0,\infty) \to [0,\infty)$ and suppose that one \textup{(}hence all\textup{)} of the conditions of Theorem~\ref{intro-5} holds.
    \begin{enumerate}[(i)]
	\item\label{intro-6-i}
    The rotationally invariant function $p_t = p_t^k$ satisfies
    \begin{equation*}
		p_t^{k}(r) = - \frac{1}{2\pi} \frac{1}{r} \frac{d}{dr} p_t^{k-2}(r) \fa r >0,\: t>0,\: k \geq 3.
	\end{equation*}
	\item\label{intro-6-ii}
    The L\'evy measure $\nu_k$ of the $k$-dimensional L\'evy process $(X_t)_{t \geq 0}$ has a rotationally invariant density $m_k$ with respect to $k$-dimensional Lebesgue measure; moreover, $m_k(\sqrt{\cdot})$ is completely monotone and satisfies
    \begin{equation*}
		m_{k}(r) = - \frac{1}{2\pi} \frac{1}{r} \frac{d}{dr} m_{k-2}(r) \fa r>0,\: k \geq 3.
	\end{equation*}
    \end{enumerate}
\end{corollary}
If we formally define an operator $T$ by $T := - \frac{1}{2\pi} \frac{1}{r} \frac{d}{dr}$, then Corollary~\ref{intro-6}\eqref{intro-6-i} reads $p_t^k = T p_t^{k-2}$. Theorem~\ref{intro-9} shows that a similar relation holds for the generator:
\begin{equation*}
	A_k u = T A_{k-2} (T^{-1} u)
\end{equation*}
for any smooth rotationally invariant function $u$ with compact support. This means that the operators $A_{k-2}$ and $A_k$ are intertwined.

\begin{theorem} \label{intro-9}
    Let $(X_t^k)_{t \geq 0}$ be a $k$-dimensional L\'evy process with rotationally invariant characteristic exponent $\psi_k(\xi) = \psi(|\xi|)$, $\xi \in \rk$, $k \geq 3$. Then there exists a $(k-2)$-dimensional L\'evy process $(X_t^{k-2})_{t \geq 0}$ with characteristic exponent $\psi_{k-2}(\xi) := \psi(|\xi|)$, $\xi \in \real^{k-2}$. If we denote by $A_{k-2}$ and $A_k$ the generator of $(X_t^{k-2})_{t \geq 0}$ and $(X_t^k)_{t \geq 0}$, respectively, then both $A_k u$ and $A_{k-2} u$ are rotationally invariant and
    \begin{equation*}
		A_k u(r)
        = \frac{1}{r} \frac{d}{dr} A_{k-2} \left( \int_0^{\bullet} s u(s) \, ds\right)(r), \quad r>0,
	\end{equation*}
	for any smooth rotationally invariant function $u$ with compact support, cf.\ \eqref{ft-eq4}.
\end{theorem}

Finally, we derive the following result on subordinated Brownian motion.
\begin{corollary} \label{intro-7}
    Let $(B_t)_{t \geq 0}$ be a $k$-dimensional Brownian motion and $(S_t)_{t \geq 0}$ a subordinator with Laplace exponent $f$. The subordinated Brownian motion $X_t := B_{S_t}$ satisfies $\Pp(X_t \in dx) = e^{-ct} \delta_0(dx) + p_t(x) \, dx$ with $c$ and $p_t$ as in \ref{intro-5}\eqref{intro-5-iv}.
    \begin{enumerate}[(i)]
	\item\label{intro-7-i} The following statements are equivalent.
        \begin{enumerate}
		\item
            $c=\infty$ \textup{(}i.e.\ $\Pp(X_t = 0)=0$ for some $t>0$\textup{)} and $\lim_{r \to 0} p_t(r)<\infty$ for small $t>0$.
		\item
            $\Ee(S_t^{-k/2})<\infty$ for small $t>0$.
		\item
            $\psi(\xi) := f(|\xi|^2)$ satisfies the Hartman--Wintner condition \eqref{lp-eq2}.
        \end{enumerate}
	\item\label{intro-7-ii}
        $c < \infty$ \textup{(}i.e.\ $\Pp(X_t = 0)>0$ for some $t>0$\textup{)} if, and only if, $(X_t)_{t \geq 0}$ is a compound Poisson process.
	\end{enumerate}
\end{corollary}

Corollary~\ref{intro-7}\eqref{intro-7-ii} implies that a subordinate Brownian motion is either a compound Poisson process or absolutely continuous for all $t>0$. Zabczyk \cite{zab} proved, more generally, that this holds for any L\'evy process with a rotationally invariant characteristic exponent. Moreover, (the proof of) Corollary~\ref{intro-7}\eqref{intro-7-i} shows that for a subordinator $(S_t)_{t \geq 0}$ with Laplace exponent $f$ we have
\begin{equation*}
	\Ee(S_t^{-\kappa})<\infty \text{\ \ for some (all) $\kappa>0$ and small (all) $t>0$}
\end{equation*}
if, and only if, $f$ satisfies the Hartman--Wintner condition $\lim_{r \to \infty} f(r)/\log r= \infty$. For a discussion of the Hartman--Wintner condition \eqref{lp-eq2} see Knopova \& Schilling \cite{knop}.

\section{Proof of Theorem~\ref{intro-5}}

In order to prove Schoenberg's original statement (the equivalence of \ref{intro-5}\eqref{intro-5-i} and \ref{intro-5}\eqref{intro-5-iii}) we will first focus on functions $f$ satisfying the Hartman--Wintner condition \begin{equation*}
	\lim_{r \to \infty} \frac{f(r)}{\log r} = \infty,
\end{equation*}
and then extend the result using an approximation argument. The key tool is the following proposition which is of independent interest. It is inspired by a publication by Kulczycki \& Ryznar \cite{kulc} where the implication ``\eqref{pr-1-ii}$\Rightarrow$\eqref{pr-1-i}'' is used to obtain gradient estimates of transition densities for L\'evy processes.

\begin{proposition} \label{pr-1}
    Let $(X_t^k)_{t \geq 0}$ be a $k$-dimensional L\'evy process with rotationally invariant characteristic exponent $\psi_k(\xi) = \psi(|\xi|)$, $\xi \in \rk$. If $\psi$ satisfies the Hartman--Wintner condition, then the following statements are equivalent.
    \begin{enumerate}[(i)]
	\item\label{pr-1-i}
        There exists a $(k+2)$-dimensional L\'evy process $(X_t^{k+2})_{t \geq 0}$ with characteristic exponent $\psi_{k+2}(\xi) := \psi(|\xi|)$, $\xi \in \real^{k+2}$.
	\item\label{pr-1-ii}
        The rotationally invariant density $p_t^k$ of $X_t^k$ satisfies $\frac{d}{dr} p_t^k(r)\leq 0$ for all $t>0$.
	\item\label{pr-1-iii}
        $X_t^k$ is unimodal isotropic for all $t>0$.
%	\item
%        The L\'evy measure of $\psi_k$ is unimodal isotropic.
	\end{enumerate}
	If one \textup{(}hence all\textup{)} of the conditions is satisfied, then
    \begin{equation}\label{pr-eq5}
		p_t^{k+2}(r) = - \frac{1}{2\pi} \frac{1}{r} \frac{d}{dr} p_t^k(r) \fa r >0.
	\end{equation}
\end{proposition}

Wolfe \cite{wolfe} and Medgyessy \cite{med} have shown that $X_t^k$ is unimodal isotropic if, and only if, the L\'evy measure $\nu_k$ is unimodal isotropic, see also Sato \cite[Theorem 54.1]{sato}. Let us briefly give an intuitive explanation for pure-jump L\'evy processes $(X_t^k)_{t \geq 0}$. It is known that the L\'evy measure $\nu_k$ is the vague limit of $t^{-1} \Pp(X_t^k \in \cdot)$ as $t \to 0$, i.e.\
\begin{equation}\label{pre-eq7}
	\nu_k(B)
    = \lim_{t \to 0} \frac 1t \Pp(X_t^k \in B)
    = \lim_{t \to 0} \frac{1}{t} \int_B p_t^k(y) \, dy
\end{equation}
for any Borel set $B \in \mathcal{B}(\mathbb{R}^k \setminus \{0\})$ with no mass at the boundary $\nu_k(\partial B)=0$, see e.g.\ \cite[Remark 6.12]{barca} or \cite[Corollary 3.3]{ihke}, and therefore $\nu_k$ is unimodal isotropic as the vague limit of unimodal isotropic distributions. On the other hand, if $\nu_k$ is unimodal isotropic, then the truncated measure $\mu_{\epsilon} := \nu_k(\cdot \cap B(0,\epsilon)^c)$ is unimodal isotropic for each $\epsilon>0$, and the associated compound Poisson process $X_t^{k,\epsilon}$ has a distribution of the form
\begin{equation*}
	e^{-\lambda_{\epsilon} t} \sum_{m=0}^{\infty} \frac{t^m}{m!} \mu_{\epsilon}^{\ast m}
\quad\text{with}\quad
    \lambda_\epsilon = \nu_k(B(0,\epsilon)^c)
\end{equation*}
which implies that $X_t^{k,\epsilon}$ is unimodal isotropic; hence, $X_t^k = \lim_{\epsilon \to 0} X_t^{k,\epsilon}$ is unimodal isotropic.

\begin{proof}[Proof of Proposition~\ref{pr-1}]
	Because of the growth condition \eqref{lp-eq2},
    \begin{equation*}
		p_t^{k+2}(x) := \frac{1}{(2\pi)^{k+2}} \int_{\real^{k+2}} e^{-i x \cdot \xi} e^{-t \psi(|\xi|)} \, d\xi
	\end{equation*}
    is well-defined and, in fact, infinitely often differentiable. Moreover, $p_t^{k+2}$ is rotationally invariant and, as $\psi(0)=0$, we have
    \begin{equation*}
		\int_{\real^{k+2}} p_t^{k+2}(x) \, dx = (\Fcal_{k+2}^{-1}p_t^{k+2})(0) = e^{-t \psi(0)} = 1.
	\end{equation*}
    Because of \eqref{ft-eq3}, the relation \eqref{pr-eq5} holds.

    \bigskip\noindent\eqref{pr-1-i}$\Rightarrow$\eqref{pr-1-ii}:
    By definition, $p_t^{k+2}$ is the density of $X_t^{k+2}$; in particular, $p_t^{k+2} \geq 0$. Because of \eqref{pr-eq5}, we get $\frac{d}{dr} p_t^k(r) \leq 0$.

    \bigskip\noindent\eqref{pr-1-ii}$\Rightarrow$\eqref{pr-1-i}:
    \eqref{pr-eq5} shows $p_1^{k+2} \geq 0$, and therefore $p_1^{k+2}$ is a density function of a probability measure, say $\mu$, on $\real^{k+2}$. By construction, we have
    \begin{equation*}
        \big( (\Fcal_{k+2}^{-1} p_{1/n}^{k+2})(\xi) \big)^n = e^{- \psi_{k+2}(\xi)} = e^{- \psi(|\xi|)} = (\Fcal_{k+2}^{-1} p_1^{k+2})(\xi) %= \widecheck{\mu}(\xi)
	\end{equation*}
    for all $n \in \nat$ and $\xi \in \real^{k+2}$. This shows that $\mu(dx) = p_1^{k+2}(x) \, dx$ is infinitely divisible. Consequently, there exists a $(k+2)$-dimensional L\'evy process $(X_t^{k+2})_{t \geq 0}$ such that $X_1^{k+2} \sim \mu$, \begin{equation*}
			\Ee e^{i \xi \cdot X_1^{k+2}} %= \widecheck{\mu}(\xi)
            = e^{-\psi_{k+2}(\xi)}.
	\end{equation*}

    \medskip\noindent\eqref{pr-1-ii}$\Leftrightarrow$\eqref{pr-1-iii}:
    Since $X_t^k$ is absolutely continuous -- due to the growth condition \eqref{lp-eq2} --, this follows directly from the definition of a unimodal isotropic distribution.
\end{proof}

\begin{remark} \label{pr-2}
    Proposition~\ref{pr-1} shows that we need an additional assumption on the growth behaviour of the density $p_t^k$ of the ($k$-dimensional) L\'evy process to ensure the existence of a L\'evy process in dimension $k+2$. This assumption is not needed to construct L\'evy processes in lower dimensions. \emph{Indeed:} Let $(X_t^k)_{t \geq 0}$ be a $k$-dimensional L\'evy process with rotationally invariant characteristic exponent $\psi_k(\xi) = \psi(|\xi|)$, and fix $d \leq k-1$. Denote by
    \begin{equation*}
		\pi_d: \rk \to \real^{d}, x = (x_1,\ldots,x_k) \mapsto (x_1,\ldots,x_{d})
	\end{equation*}
	the projection onto the first $d$ coordinates. Since
    \begin{equation*}
		\Ee e^{i \xi \cdot \pi_d(X_t^k)}
		= \Ee e^{i \tilde{\xi} \cdot X_t^k}
		= e^{-t \psi(|\tilde{\xi}|)}
		= e^{-t \psi(|\xi|)}, \quad \xi  \in \real^d,
	\end{equation*}
    for $\tilde{\xi} := (\xi,0,\ldots,0) \in \rk$, it is not difficult to see that $X_t^{d} := \pi_{d}(X_t^k)$ defines a $d$-dimensional L\'evy process with characteristic exponent $\psi_{d}(\xi) = \psi(|\xi|)$, $\xi \in \real^d$.
\end{remark}

Proposition~\ref{pr-1} can be used to derive gradient estimates for the semigroup; they are not needed for the proof of Schoenberg's theorem but are of independent interest.
\begin{corollary} \label{pr-25}
    Let $(X_t)_{t \geq 0}$ be a one-dimensional L\'evy process with characteristic exponent $\psi(\xi) = f(|\xi|^2)$ for some Bernstein function $f$. If $f$ satisfies the Hartman--Wintner condition, then the semigroup $P_t u(x) := \Ee u(x+X_t) = \int u(x+y) p_t(y) \, dy$ satisfies the gradient estimate
    \begin{equation}\label{pr-eq6}
		\left|\frac{d}{dx} P_t u(x) \right| \leq 4 \|u\|_{\infty} \|p_t\|_{\infty}
	\end{equation}
	for all bounded Borel measurable functions $u: \real \to \real$.
\end{corollary}

\begin{remark}
\begin{enumerate}[(i)]
\item
    Using a convolution argument, it is not difficult to extend Corollary~\ref{pr-25} to L\'evy processes whose characteristic exponent $\psi$ satisfies $\psi(\xi) \geq c f(|\xi|^2)$ for some Bernstein function $f$ and $c>0$; see \cite[Lemma~3.1]{berger}.
\item
    Clearly, $
		\|p_t\|_{\infty} \leq \int_{\real} e^{-t \psi(\xi)} \, d\xi
	$; we can estimate the integral if we have additional information on the growth of $f$, see e.g.\ \cite{ssw}.
\item
    In \cite[Lemma~4.5]{maruyama}, \eqref{pr-eq5} is used to obtain gradient estimates in terms of moments.
\item
    It is possible to iterate \eqref{pr-eq6} to derive estimates for derivatives of higher order, cf. \cite[Lemma~4.1]{maruyama} for details.
\end{enumerate}
\end{remark}

\begin{proof}[Proof of Corollary~\ref{pr-25}]
    It follows from Proposition~\ref{pr-1} that $(X_t)_{t \geq 0}$ is unimodal isotropic, and therefore the density $p_t$ (which exists because of \eqref{lp-eq2}) is unimodal. Using exactly the same reasoning as in
    \cite[proof of Theorem 3.4]{berger} we conclude that
    \begin{equation*}
		\int_{\real} |p_t(x+y)-p_t(x)| \, dx
        \leq 4 |y| \, \|p_t\|_{\infty},
        \quad t>0.
	\end{equation*}
	Applying Fatou's lemma we get
    \begin{equation*}
        \int_{\real} |p_t'(x)| \, dx
        \leq \liminf_{|y| \to 0} \int_{\real} \left| \frac{p_t(x+y)-p_t(x)}{y} \right| \, dx
        \leq 4 \|p_t\|_{\infty}.
	\end{equation*}
	Since
    \begin{equation*}
		\Ee u(x+X_t)
        = \int_{\real} u(x+y) p_t(y) \, dy
        = \int_{\real} u(y) p_t(y-x) \, dy,
	\end{equation*}
    a variant of the differentiation lemma for parameter-dependent integrals, cf.\ \cite[Proposition A.1]{maruyama} or \cite[Problem~14.20]{mims}, yields
    \begin{equation*}
		\frac{d}{dx} \Ee u(x+X_t) = -\int_{\real} u(y) p_t'(y-x) \, dy.
	\end{equation*}
	Note that the differentiation lemma is indeed applicable since the map
	\begin{equation*}
		x \mapsto \int_{\real} u(y) p_t'(y-x) \, dy
	\end{equation*}
    is continuous: it is the convolution of a Lebesgue-integrable function with a bounded function, see e.g.\ \cite[Theorem~15.8(ii)]{mims}. 	Hence,
    \begin{equation*}
        \left| \frac{d}{dx} \Ee u(x+X_t) \right| \leq \|u\|_{\infty} \int_{\real} |p_t'(y-x)| \, dy
        \leq 4 \|u\|_{\infty} \|p_t\|_{\infty}. \qedhere
	\end{equation*}
\end{proof}

We are now ready to prove the first part of Schoenberg's theorem.
\begin{proof}[Proof of Theorem~\ref{intro-5}] We will prove the equivalence of \eqref{intro-5-i}, \eqref{intro-5-ii} and \eqref{intro-5-iii}.

\medskip\noindent
The direction \eqref{intro-5-i}$\Rightarrow$\eqref{intro-5-ii} is clear. For \eqref{intro-5-ii}$\Rightarrow$\eqref{intro-5-iii} we assume first that $f$ satisfies the Hartman--Wintner condition, i.e.\ $\lim_{r \to \infty} f(r)/\log r = \infty$. By assumption, $\psi_k(\xi) := \psi(|\xi|) := f(|\xi|^2)$, $\xi \in \real^{k}$, is a continuous negative definite function for $k =1+2n$, $n \in \nat_0$, satisfying \eqref{lp-eq2}. In particular, there exists a $k$-dimensional L\'evy process $(X_t^k)_{t \geq 0}$ with characteristic exponent $\psi_{k}$. Because of \eqref{lp-eq2}, $X_t^k$ has  a density $p_t^k$ with respect to Lebesgue measure. In particular, \eqref{pr-1-i} in Proposition~\ref{pr-1} holds for any $k =1+2n$. If we set $g_t^{k}(r) := p_t^k(2\sqrt{r})$, then by \eqref{pr-eq5},
\begin{align*}
		\frac{d}{dr} g_t^k(r)
		= \frac{1}{\sqrt{r}} \frac{d}{ds} p_t^k(s) \bigg|_{s=2\sqrt{r}}
		\stackrel{\eqref{pr-eq5}}{=} \frac{1}{\sqrt{r}} (-2\pi s p_t^{k+2}(s)) \bigg|_{s=2\sqrt{r}}
		= - 4\pi g_t^{k+2}(r).
\end{align*}
Iterating this procedure, we obtain
\begin{equation*}
		\frac{d^n}{dr^n} g_t^1(r) = (-4\pi)^n g_t^{1+2n}(r) \fa n \in \mathbb{N}, r >0.
\end{equation*}
As $g_t^{1+2n} \geq 0$ for $n \in \nat_0$, this proves that $g_t^1$ is completely monotone, i.e.\ there exists a finite measure $\mu_t$ on $(0,\infty)$ such that
\begin{equation*} 
    g_t^1(r) = \int_{(0,\infty)} e^{-rs} \, \mu_t(ds), \quad r \geq 0.
\end{equation*}
Applying Fubini's theorem we find
\begin{equation}\label{pr-eq7}
\begin{aligned}
		e^{-t f(r^2)}
		= \int_{\real} e^{ix r} p_t^1(|x|) \, dx
		&= \int_{\real} e^{ix r} \left( \int_{(0,\infty)} e^{-s |x|^2/4} \mu_t(ds) \right) \, dx \\
		&= \int_{(0,\infty)} \int_{\real} e^{ix r} e^{-s |x|^2/4} \, dx \, \mu_t(ds) \\
		&= \sqrt{4\pi} \int_{(0,\infty)} \frac{1}{\sqrt{s}} e^{-r^2/s} \, \mu_t(ds).
\end{aligned}
\end{equation}
This identity shows that $r \mapsto e^{-t f(r)}$ is the Laplace transform of a finite measure, hence, completely monotone. In view of \eqref{lp-eq5} this implies that
\begin{equation*}
		-\frac{1}{t} \frac{d}{dr} e^{-t f(r)} = f'(r) e^{-tf(r)}
\end{equation*}
is completely monotone. Letting $t \to 0$ we conclude that $f'$ is completely monotone, and so $f$ is a Bernstein function; see \cite[Theorem 3.7]{bernstein} for an alternative proof that $e^{-t f(\cdot)} \in \cm$ implies that $f$ is a Bernstein function.

If $f$ does not satisfy the Hartman--Wintner condition, we set
\begin{equation*}
		f_{\epsilon}(r) := f(r) + \epsilon r.
\end{equation*}
Note that $\rk \ni \xi \mapsto f_{\epsilon}(|\xi|^2)$ is a continuous negative definite function for any $k=1+2n$, $n \in \nat_0$ and $\epsilon>0$. As $f \geq 0$, $f_{\epsilon}$ obviously satisfies the Hartman--Wintner condition. The first part of this proof shows that $f_{\epsilon}$ is a Bernstein function. Consequently, $f$ is a Bernstein function as the pointwise limit of Bernstein functions, cf.\ \cite[Corollary 3.8(ii)]{bernstein}.

\bigskip\noindent
\eqref{intro-5-iii}$\Rightarrow$\eqref{intro-5-i}:
Since $f$ is a Bernstein function, there exists a subordinator $(S_t)_{t \geq 0}$ with Laplace transform $e^{-tf}$. If $(B_t)_{t \geq 0}$ is a $k$-dimensional Brownian motion, then $X_t := B_{S_t}$ is a $k$-dimensional L\'evy process with characteristic exponent $f(|\xi|^2)$. This implies that $\rk \ni \xi \mapsto f(|\xi|^2)$ is a continuous and negative definite function.
\end{proof}

It remains to prove the equivalence of \eqref{intro-5-iii} and \eqref{intro-5-iv} in Theorem~\ref{intro-5}. To this end, we recall a result on the distribution of subordinated Brownian motion.
\begin{lemma} \label{pr-3}
    Let $(B_t)_{t \geq 0}$ be a $k$-dimensional Brownian motion and $(S_t)_{t \geq 0}$ a subordinator with Laplace exponent $f$. Then the distribution of $X_t := B_{S_t}$ equals
    \begin{equation*}
		\Pp(X_t \in B) = \Pp(S_t = 0) \delta_0(B)  + \int_B p_t(x) \, dx, \quad B \in \Bcal(\rk),\; t>0;
	\end{equation*}
	here $p_t: \rk \to [0,\infty)$ is a rotationally invariant function and $p_t(\sqrt{\cdot}) \in \cm$.
\end{lemma}

\begin{proof}
	Since $(S_t)_{t \geq 0}$ and $(B_t)_{t \geq 0}$ are independent, we have
    \begin{align*}
		\Pp(X_t \in B)
		&= \int_{[0,\infty)} \Pp(B_s \in B) \, \Pp(S_t \in ds) \\
        &= \delta_0(B) \Pp(S_t = 0) + \int_{(0,\infty)} \left( \int_{B} \frac{1}{(2\pi s)^{k/2}} \exp \left(- \frac{|y|^2}{2s} \right) dy \right)  \Pp(S_t \in ds) \\
		&= \delta_0(B) \Pp(S_t = 0) + \int_B p_t(y) \, dy
	\end{align*}
	for
    \begin{equation}\label{pr-eq9}
		p_t(y) := \int_{(0,\infty)} \frac{1}{(2\pi s)^{k/2}} \exp \left(- \frac{|y|^2}{2s} \right) \, \Pp(S_t \in ds). 	
    \end{equation}
    The fact that $p_t(\sqrt{\cdot})$ is completely monotone follows directly from this representation and the differentiation lemma for parameter-dependent integrals.
\end{proof}

\begin{proof}[Proof of Theorem~\ref{intro-5}, equivalence of \eqref{intro-5-iii} \& \eqref{intro-5-iv}]
\eqref{intro-5-iii}$\Rightarrow$\eqref{intro-5-iv}: As $f$ is a Bernstein function, there exists a subordinator $(S_t)_{t \geq 0}$ with Laplace exponent $f$. The subordinated Brownian motion $X_t := B_{S_t}$ is a L\'evy process with characteristic exponent $\psi(\xi) = f(|\xi|^2)$ and, by Lemma~\ref{pr-3},
\begin{equation*}
		\Pp(X_t \in B) = \Pp(S_t = 0) \delta_0(B)  + \int_B p_t(x) \, dx, \quad B \in \Bcal(\rk),
\end{equation*}
for the rotationally invariant non-negative function $p_t$ defined in \eqref{pr-eq9} satisfying $p_t(\sqrt{\cdot}) \in \cm$. By the Markov property of $(S_t)_{t \geq 0}$, we have for any $s \leq t$
\begin{equation*}
		\Pp(S_t=0)
		= \Ee\left[\Pp(z+S_{t-s}=0) \big|_{z=S_s} \right]
		= \Ee\left[\I_{\{S_s=0\}} \Pp(S_{t-s}=0)\right]
\end{equation*}
as $(S_t)_{t \geq 0}$ has non-decreasing sample paths. Consequently, $f(t) := \Pp(S_t=0)$ satisfies $f(t+s) = f(t) f(s)$. Since $f$ is right-continuous, this implies, by the Cauchy--Abel functional equation, $\Pp(S_t=0)= f(t) = e^{-ct}$ for some $c \in [0,\infty]$, see e.g.\ \cite[Theorem A.1]{barca} for a proof.

\bigskip\noindent
For the converse direction	\eqref{intro-5-iv}$\Rightarrow$\eqref{intro-5-iii}, we can argue exactly as in the proof of \eqref{pr-eq7} in the step ``\eqref{intro-5-ii}$\Rightarrow$\eqref{intro-5-iii}.''
\end{proof}

\section{Proof of Corollary~\ref{intro-6}, Theorem~\ref{intro-9} \& Corollary~\ref{intro-7}}

\begin{proof}[Proof of Corollary~\ref{intro-6}]
    \eqref{intro-6-i}
    The proof of Theorem~\ref{intro-5} shows that $p_t^k$ is given by \eqref{pr-eq9}, i.e.
    \begin{equation*}
    		p_t^k(r) = \int_{(0,\infty)} \frac{1}{(2\pi s)^{k/2}} \exp \left(- \frac{r^2}{2s} \right) \, \Pp(S_t \in ds).
    \end{equation*}
    If we differentiate $p_t^k$ with respect to $r$, then we obtain
    \begin{align*}
    		\frac{d}{dr} p_t^k(r)
    		= - r \int_{(0,\infty)} \frac{1}{s} \frac{1}{(2\pi s)^{k/2}} \exp \left(- \frac{r^2}{2s} \right) \Pp(S_t \in ds)
    		= - 2\pi r p_t^{k+2}(r).
    \end{align*}

    \medskip\noindent\eqref{intro-6-ii}
    Since the L\'evy measure $\nu_k(dy)$ is the vague limit of $t^{-1} p_t^k(y) \, dy$, cf.\ \eqref{pre-eq7}, the assertion follows formally from \eqref{intro-6-i} by dividing both sides by $t^{-1}$ and letting $t \to 0$. For a rigorous argument we note that the density $m_k$ of the L\'evy measure $\nu_k$ is given by
    \begin{equation*}
    		m_k(r) = \int_{(0,\infty)} \frac{1}{(2\pi s)^{k/2}} \exp \left(- \frac{r^2}{2s} \right) \, \mu(ds)
    \end{equation*}
    where $\mu$ denotes the L\'evy measure of the subordinator $(S_t)_{t \geq 0}$, cf.\ \cite[Theorem 30.1]{sato}. Now the claim follows using exactly the same calculation as in \eqref{intro-6-i}.
\end{proof}

\begin{proof}[Proof of Theorem~\ref{intro-9}]
    For the existence of the process $(X_t^{k-2})_{t \geq 0}$ see Remark~\ref{pr-2}. Since both $A_k$ and $A_{k-2}$ are pseudo-differential operators with rotationally invariant symbols, cf.\ \eqref{lp-eq1}, it is obvious that $A_k u$ and $A_{k-2} u$ are rotationally invariant for any smooth rotationally invariant $u$ with compact support. By \eqref{ft-eq3} and \eqref{lp-eq1}, we have
    \begin{align*}
		A_k u(r)
		= - \Fcal_k^{-1}( \psi \cdot \Fcal_k u)(r)
		&=   2\pi \frac{1}{r} \frac{d}{dr} \Fcal_{k-2}^{-1}(\psi \cdot \Fcal_k u)(r) \\
		&= -  2\pi \frac{1}{r} \frac{d}{dr} A_{k-2}(\Fcal_{k-2}^{-1} \Fcal_k u)(r).
	\end{align*}
	If we can show that
    \begin{equation*}
		-  2\pi \Fcal_{k-2}^{-1}(\Fcal_k u)(r) = \int_0^r s \cdot u(s) \, ds -C =: v(r), \quad r>0,
	\end{equation*}
    where $C := \int_0^{\infty} s u(s) \, ds$, then the claim follows; note that $A_{k-2}(C\I_{\real^{k-2}})=0$ by the very definition of the generator.  Applying \eqref{ft-eq2}, we find
    \begin{align*}
		\Fcal_k u(r)
		&= \frac{1}{(2\pi)^{k/2} r^{k/2-1}} \int_{0}^{\infty} u(s) s^{k/2} J_{k/2-1}(s r) \, ds \\
		&= \frac{1}{(2\pi)^{k/2} r^{k/2-1}} \int_0^{\infty}  \left( \frac{d}{ds} v(s) \right) s^{k/2-1} J_{k/2-1}(s r) \, ds.
	\end{align*}
	Since $v$ has compact support, the integration by parts formula yields \begin{align*}
		\Fcal_k u(r)
		&= - \frac{1}{(2\pi)^{k/2} r^{k/2-1}} \int_0^{\infty} v(s) s^{k/2-2} \left( \left[ \frac{d}{2}-1 \right] J_{k/2-1}(sr) + s \frac{d}{ds} J_{k/2-1}(sr) \right) \, ds.
	\end{align*}
	As \begin{equation*}
		\frac{d}{dz} J_{k/2-1}(z) = - \left( \frac{d}{2}-1 \right) \frac{1}{z} J_{k/2-1}(z) + J_{k/2-2}(z),
	\end{equation*}
	cf.\ \cite[(10.6.2)]{nist}, we get \begin{align*}
		\Fcal_k u(r)
		&= - \frac{1}{(2\pi)^{k/2} r^{(k-2)/2-1}} \int_0^{\infty} v(s) s^{(k-2)/2} J_{(k-2)/2-1}(sr) \, du
		= - \frac 1{2\pi} \Fcal_{k-2} v(r).
	\end{align*}
	Consequently, \begin{equation*}
		\Fcal_{k-2}^{-1}(\Fcal_k u)(r) = - \frac 1{2\pi} v(r).\qedhere
	\end{equation*}
\end{proof}

\begin{proof}[Proof of Corollary~\ref{intro-7}]
\eqref{intro-7-i}
    By \eqref{pr-eq9} and the monotone convergence theorem, we have
    \begin{equation*}
		\lim_{r \to 0} p_t(r)
		=  \int_{(0,\infty)} \frac{1}{(2\pi s)^{k/2}} \, \Pp(S_t \in ds)
		= \frac{1}{(2\pi)^{k/2}} \Ee(S_t^{-k/2} \I_{\{S_t \neq 0\}}).
	\end{equation*}
	Moreover,
    \begin{equation*}
		\Pp(S_t=0) = \Pp(X_t=0) = e^{-ct}.
	\end{equation*}
    From this the equivalence ``(a)$\Leftrightarrow$(b)'' follows easily. In order to prove ``(b)$\Leftrightarrow$(c)'' we use the following elementary identity
    \begin{equation*}
		\frac{1}{y^\kappa} = \frac{1}{\Gamma(\kappa)} \int_0^{\infty} e^{-ry} r^{\kappa-1} \, dr,
        \quad y \geq 0,\; \kappa>0
	\end{equation*}
	which entails
    \begin{equation*}
		\Ee(S_t^{-k/2}) = \frac{1}{\Gamma(k/2)} \int_0^{\infty} e^{-t f(r)} r^{k/2-1} \, dr.
	\end{equation*}
    If $\psi(\xi) := f(|\xi|^2)$ satisfies the Hartman--Wintner condition, then obviously $\Ee(S_t^{-k/2})<\infty$ for all $t >0$. Conversely, suppose that $\Ee(S_t^{-k/2})<\infty$ for sufficiently small $t>0$. Introducing polar coordinates, we get \begin{align*}
		\int_1^\infty e^{-t f(r)} r^{-1/2} \, dr
        \leq \int_0^\infty e^{-t f(r)} r^{k/2-1} \, dr < \infty,
	\end{align*}
	and so
    \begin{align*}
		\infty> \int_1^\infty e^{-t f(r)} r^{-1/2} \, dr
		= 4 \int_1^\infty e^{-t f(s^4)} s \, ds.
	\end{align*}
    Since $s \mapsto e^{-t f(s^4)}$ is a continuous function this implies $\lim_{s \to \infty} e^{-t f(s^4)} s = 0$. This, in turn, gives \begin{equation*}
		\infty = -\lim_{s \to \infty} \log \left(e^{-t f(s^4)} s \right) = \lim_{s \to \infty} (t f(s^4) - \log s).
	\end{equation*}
	Consequently, \begin{equation*}
		\frac{f(s^4)}{\log (s^2)} \geq \frac{1}{t} \frac{\log s}{2 \log s} = \frac{1}{2t}
	\end{equation*}
    for $s>0$ sufficiently large and $t>0$ small. Letting $t \to 0$ proves $\lim_{s \to \infty} \frac{f(s^4)}{\log(s^2)} = \infty$, and this implies readily the assertion.

\bigskip\noindent\eqref{intro-7-ii} The direction ``$\Leftarrow$'' follows directly from the definition of a compound Poisson process. To prove ``$\Rightarrow$'', we define a stopping time $\tau := \inf\{t>0; X_t \neq 0\}$. By the strong Markov property,
\begin{equation*}
		\Pp(X_t = 0, \tau < t)
		= \Ee\left[\I_{\{\tau<t\}} \Pp(z+X_{t-\tau}=0) \big|_{z=X_{\tau}} \right]
		= 0
\end{equation*}
as $\Pp(x+X_s = 0)=0$ for all $s>0$ and $x \neq 0$, cf.\ \eqref{intro-eq11}. This shows
\begin{equation*}
		\Pp(X_t = 0)
		= \Pp(X_t = 0, \tau \geq t)
		= \Pp \left( \sup_{s \leq t} |X_s| = 0 \right).
\end{equation*}
On the other hand, we have
\begin{equation*}
		\Pp \left( \sup_{s \leq t} |X_s| = 0 \right)
		\leq \Pp(N_t(\rk \setminus \{0\})=0)
		= e^{-t \nu(\rk \setminus \{0\})};
\end{equation*}
here $N$ denotes the jump measure \eqref{lp-eq3} and $\nu$ the L\'evy measure of $(X_t)_{t \geq 0}$. Combining both considerations and using that, by assumption and \eqref{intro-eq11}, $\Pp(X_t = 0)= e^{-ct}$ for some $c \in [0,\infty)$, we get $\nu(\rk \setminus \{0\}) = c <\infty$. This proves that $(X_t)_{t \geq 0}$ is a compound Poisson process.
\end{proof}

For an alternative proof of Corollary~\ref{intro-7}\eqref{intro-7-ii} see \cite[Theorem 27.4]{sato}.

\end{document}